\documentclass[12pt]{article}
\usepackage{amssymb}
\usepackage{amsmath}
\usepackage{amsthm}
\usepackage{amsfonts}

\usepackage{a4wide}
\usepackage{longtable}
\usepackage{multirow}

\usepackage[latin1]{inputenc}
\usepackage[english]{babel}
\usepackage[T1]{fontenc}
\usepackage{amssymb,amsmath,amsthm,amstext,amsfonts,latexsym}

\usepackage{pgf,tikz}
\usepackage{mathrsfs}
\usetikzlibrary{arrows}

\newtheorem{lemma}{Lemma}[section]
\newtheorem{proposition}{Proposition}[section]
\newtheorem{theorem}{Theorem}[section]
\newtheorem{corollary}{Corollary}[theorem]

\theoremstyle{definition}
\newtheorem{definition}{Definition}[section]

\theoremstyle{remark}
\newtheorem{remark}{Remark}[section]

\usepackage[all]{xy}

\input xy
\xyoption {all}
\usepackage[pdftex,         %
    bookmarks=true,         
    hypertexnames=false,        
    bookmarksnumbered=true,     
    pdfpagemode=None,       
    pdfstartview=FitH,      
    pdfpagelayout=SinglePage,   
    colorlinks=true,        
    urlcolor=blue,          
    citecolor=blue,      
    pdfborder={0 0 0}       
    ]{hyperref}

\begin{document}

\title{ The generators of $3$-class group of some fields of degree $6$ over $\mathbb{Q}$ 
}

\author{Siham AOUISSI, Moulay Chrif ISMAILI, Mohamed TALBI and Abdelmalek AZIZI}

\maketitle
 
\medskip  \noindent \textbf{Abstract:}
$\\ $ Let $\mathrm{k}=\mathbb{Q}\left(\sqrt[3]{p},\zeta_3\right)$, where $p$ is a  prime number such that $p \equiv 1 \pmod  9$, and let
$C_{\mathrm{k},3}$ be the $3$-component of the class group of $\mathrm{k}$.  In \cite{GERTH3}, Frank Gerth III proves a conjecture made by Calegari and Emerton \cite{Cal-Emer}  which gives  necessary and sufficient conditions for
 $C_{\mathrm{k},3}$ to be of $\operatorname{rank}\,$ two. The purpose of the present work is to determine generators of $C_{\mathrm{k},3}$, whenever it is isomorphic to $\mathbb{Z}/9\mathbb{Z} \times \mathbb{Z}/3\mathbb{Z}$.
\bigskip

 \noindent{\bf Keywords:} {Pure cubic  fields, $3$-class groups, generators.\\
 AMSC: 11R11, 11R16, 11R20, 11R27, 11R29, 11R37.}\bigskip

\section{Introduction}\label{sec1}
Let $\Gamma=\mathbb{Q}(\sqrt[3]{p})$ be a pure cubic field, where $p$ is a prime number such that $p \equiv 1  \pmod  9$. We denote by $\zeta_3=-1/2+i\sqrt{3}/2$ the normalized primitive third roots of unity,  $\mathrm{k}=\mathbb{Q}(\sqrt[3]{p},\zeta_3)$ the normal closure of $\Gamma$ and  $C_{\mathrm{k},3}$ the $3$-component of the class group of $\mathrm{k}$.\\
Assuming $9$ divides exactly the $3$-class number of $\Gamma$. Then $C_{\mathrm{k},3} \simeq \mathbb{Z}/9\mathbb{Z} \times \mathbb{Z}/3\mathbb{Z}$ if and only if $u=1$, where $u$ is an index of units that will be defined in the notations below. In this paper, we will determine the generators of $C_{\mathrm{k},3}$ when $C_{\mathrm{k},3}$ is of type $(9,3)$ and $3$ is not a cubic residue  modulo $p$. We spot that Calegari and Emerton (\cite[Lemma 5.11]{Cal-Emer}) proved that the $\operatorname{rank}\,$ of the $3$-class group of $\mathbb{Q}(\sqrt[3]{p},\zeta_3)$, with $p \equiv 1  \pmod  9$, is equal to two if $9$ divides the $3$-class number of $\mathbb{Q}(\sqrt[3]{p})$. Moreover, in his work \cite[Theorem 1, p.471]{GERTH3}, Frank Gerth III proves that the converse to Calegari-Emerton's  result   is also true. The present work can be viewed as a continuation of the works \cite{Cal-Emer} and \cite{GERTH3}  .\\  
After reviewing some basic properties of the norm residue symbols and prime factorization in the normal closure of a pure cubic field that will be needed later, we will establish in section \ref{Sect3} some preliminary results of the $3$-class group $C_{\mathrm{k},3}$. Using this, we arrive to determine the generators of $3$-class groups $C_{\mathrm{k},3}$ of type $(9,3)$. All the study cases are illustrated by numerical examples and summarized in tables in section \ref{sect4}.  The usual notations on which the work is based is as follows:
\begin{itemize}
 \item $\Gamma= \mathbb{Q}(\sqrt[3]{d})$ : a pure cubic field, where $d$ is a cube-free natural number;
 \item $\mathrm{k}_0= \mathbb{Q}(\zeta_3)$ : the third cyclotomic field ; 
  \item $\mathrm{k}=\mathbb{Q}(\sqrt[3]{d},\zeta_3)$ : the normal closure of the pure cubic field $\Gamma$;
  \item $u=[E_{\mathrm{k}}: E_0]$ : the index of the sub-group $E_0$ generated by the units of intermediate fields of the extension  $\mathrm{k}/\mathbb{Q}$ in $E_{k}$ the group of units of $\mathrm{k}$;
 \item $\lambda= 1-\zeta_3$ prime integer of $\mathrm{k}_0$;
  \item $\langle\tau \rangle=\operatorname{Gal}\left(\mathrm{k}/\Gamma\right)$, $\tau^2=id, \tau(\zeta_3)=\zeta_3^2$ and $\tau(\sqrt[3]{d})=\sqrt[3]{d}$;
   \item $ \langle \sigma \rangle=\operatorname{Gal}\left(\mathrm{k}/\mathbb{Q}(\zeta_3)\right)$, $\sigma^3=id, \sigma(\zeta_3)=\zeta_3$ and $\sigma(\sqrt[3]{d})=\zeta_3\sqrt[3]{d}$;
  \item For an algebraic number field $L$:
  \begin{itemize}
  \item $\mathcal{O}_{L}$ : the ring of integers of $L$;
  \item $E_{L}$ : the group of units of $L$;
   \item $\mathcal{D}_{L}$ : the discriminant of $L$;
   \item $h_{L}$ : the class number of $L$;
   \item $h_{L,3}$ : the $3$-class number of $L$;
 \item $C_{L,3}$ : the $3$-class group of $L$;
 \item $L_3^{(1)}$ : the Hilbert $3$-class field of $L$;
 \item $[\mathcal{I}]$ : the class of a fractional ideal $\mathcal{I}$  in the class group of $L$;
  \end{itemize}
 \item $\left(\dfrac{c}{p}\right)_3 = 1 \Leftrightarrow X^3\equiv c (mod~p)$ resolved on $\mathbb{Z}
\Leftrightarrow c^{(p-1)/3} \equiv 1 \pmod p$, where $c \in \mathbb{Z}$ and $p$ is a prime number congruent to $1 (mod~3)$.
 \end{itemize}

\section{Norm residue symbol and ideal factorization theory}\label{sec2}
\subsection{The norm residue symbol}
$\\ $ Let $L/K$ an abelian extension of number fields with conductor $f$. For each finite or infinite prime ideal $\mathcal{P}$ of $\mathrm{K}$, we note by $f_{\mathcal{P}}$ the largest power of $\mathcal{P}$ that divides $f$. Let $a \in K^{*}$, we determine an auxiliary number $a_0$ by the two conditions $a_0 \equiv a ~ (mod~ f_{\mathcal{P}})$ and $a_0 \equiv 1 ~ (mod~ \frac{f}{f_{\mathcal{P}}})$. Let $\mathcal{Q}$ an ideal co-prime with $\mathcal{P}$ such that $(a_0)=\mathcal{P}^{e}\mathcal{Q}$ ($b=0$ if $\mathcal{P}$ is infinite). We note by
$$\left(\frac{a,
L}{\mathcal{P}}\right)=\left(\frac{L/K}{\mathcal{Q}}\right)$$
the Artin map in $L/K$ applied to  $\mathcal{Q}$.

\begin{definition} 
Let $K$ be a number field containing the $l^{th}$-roots of units, where $l\in \mathbb{N}$, then for each $a,b \in K^{*}$ and prime ideal $\mathcal{P}$ of $K$, we define the \textbf{\textit{norm residue symbol}} by:
 
$$\left(\dfrac{a,
b}{\mathcal{P}}\right)_{l}=\dfrac{\left(\dfrac{a, K(\sqrt[l]{b})}{\mathcal{P}}\right) \sqrt[l]{b} }{\sqrt[l]{b}}.$$
Therefore, if the prime ideal $\mathcal{P}$ is unramified
in the field $ K(\sqrt[l]{b})$, then we write
 $$\left(\dfrac{
b}{\mathcal{P}}\right)_{l}=\dfrac{\left(\dfrac{ K(\sqrt[l]{b})}{\mathcal{P}}\right) \sqrt[l]{b} }{\sqrt[l]{b}}.$$

\end{definition}

\begin{remark} Notice that $ \left(\dfrac{a,b}{\mathcal{P}}\right)_{l} $ and $\left(\dfrac{b}{\mathcal{P}}\right)_{l}$
are two $l^{th}$-roots of units.
\end{remark}
Following \cite{Hass2}, the principal properties of the norm residue symbol are given as follows:
\subsection*{Properties}
\begin{enumerate}
\item The product formula:
\begin{itemize}

\item $ \left(\dfrac{a_1 a_2,b}{\mathcal{P}}\right)_{l} = \left(\dfrac{a_1,b}{\mathcal{P}}\right)_{l} \left(\dfrac{a_2,b}{\mathcal{P}}\right)_{l}$;

\item $ \left(\dfrac{a,b_1b_2}{\mathcal{P}}\right)_{l} = \left(\dfrac{a,b_1}{\mathcal{P}}\right)_{l} \left(\dfrac{a,b_2}{\mathcal{P}}\right)_{l}$;
\end{itemize}

\item The inverse formula: $ \left(\dfrac{a,b}{\mathcal{P}}\right)_{l}= \left(\dfrac{b,a}{\mathcal{P}}\right)_{l}^{-1};$

\item $\left(\dfrac{a,b}{\mathcal{P}}\right)_{l}=1 \Leftrightarrow $ $a$ is norm residue of $K(\sqrt[l]{b})$ modulo $f_{b}$;

\item $ \left(\dfrac{\sigma a,\sigma b}{\sigma\mathcal{P}}\right)_{l}=\sigma \left(\dfrac{a,b}{\mathcal{P}}\right)_{l},$ for each automorphism $\sigma$ of $K$; 

\item If $\mathcal{P} $ is not divisible by the conductor $f_{b}$ of $K(\sqrt[l]{b})$ and appears in $(a)$ with the exponent $e$, then:

\begin{itemize}
\item $ \left(\dfrac{a,b}{\mathcal{P}}\right)_{l}= \left(\dfrac{b}{\mathcal{P}}\right)_{l}^{-e};$

\item $\mathcal{P}$ is infinite $(e=0)$  $\Rightarrow \left(\dfrac{a,b}{\mathcal{P}}\right)_{l}=1;$
\end{itemize}

\item The classical reciprocity law: let $a,b\in K^{*}$, and the conductors $f_{a}$ and $f_{b}$ of respectively $K(\sqrt[l]{a})$ and $K(\sqrt[l]{b})$ are co-prime, then:
$$ \left(\dfrac{a}{(b)}\right)_{l} =\left(\dfrac{b}{(a)}\right)_{l};$$ 

\item $ \prod_{\mathcal{P}} \left(\dfrac{a,b}{\mathcal{P}}\right)_{l}=1$, where the product is taken on the finite and infinite prime ideals;

\item Let $L$ is a finite extension of $K$, $a \in L$ and $b \in K^{*}$, then:
$$ \prod_{\mathfrak{P}|\mathcal{P}} \left(\dfrac{a,b}{\mathfrak{P}}\right)_{l}=  \left(\dfrac{\mathcal{N}_{L/K}(a),b}{\mathcal{P}}\right)_{l}. $$

\end{enumerate} 

\begin{remark}\label{RqMes}
From property $(3)$, we have:
$$ a \text{ is a norm  in} \  K(\sqrt[l]{b}) \Rightarrow \left(\dfrac{a,b}{\mathcal{P}}\right)_{l}=1, $$ 
for each prime ideal $\mathcal{P}$ of $K$.
\end{remark}

For more basic properties of the norm residue symbol in the number fields, we refer the reader to the papers \cite{DED},  \cite{Hass1} and \cite{Hass2}. Notice that in section \ref{Sect3}, we will use the norm cubic residue symbols $(l=3)$. As the ring of integer $ \mathcal{O}_{\mathrm{k}_0}$ is principal, $h_{\mathrm{k}_0} = 1,$ 
we will write the norm cubic residue symbol as follows:
\begin{center}
$\left(\dfrac{a, b}{(\pi)}\right)_3 = \left(\dfrac{a, b}{\pi}\right)_3 $ and
$\left(\dfrac{a}{(\pi)}\right)_3 = \left(\dfrac{a}{\pi}\right)_3$
\end{center}
 where           
$a, b \in \mathrm{k}_0^*$ and $\pi$ is a prime integer of $\mathcal{O}_{\mathrm{k}_0}.$


\subsection{Prime factorization in a pure cubic field and in its normal closure}\label{subfactor}
 $\\ $
 Let be $\Gamma=\mathbb{Q}(\sqrt[3]{d})$ a pure cubic field, and  $\mathcal{O}_{\Gamma}$ the ring of integers of $\Gamma$. We write the natural integer $d$ in form $d=ab^2,$ where $a$ and $b$ are cube-free and co-prime positive integers. In his paper \cite{DED},  Dedekind has defined two different types of pure cubic fields as follows: 
\begin{definition}\label{3inG}
Using the same notations as above:
\begin{enumerate}
\item We say that $\Gamma=\mathbb{Q}(\sqrt[3]{d})$ is  of the \textit{\textbf{first kind}} if $3\mathcal{O}_{\Gamma}=\mathcal{P}^3$, where $\mathcal{P}$ is a prime ideal of $\mathcal{O}_{\Gamma}$, in this case, $a^2-b^2\not\equiv 0 \ (mod~9).$
\item We say that $\Gamma=\mathbb{Q}(\sqrt[3]{d})$ is of the \textit{\textbf{second kind}} if $3\mathcal{O}_{\Gamma}=\mathcal{P}^2\mathcal{P}_1$, where $\mathcal{P}\neq \mathcal{P}_1$ are two primes of $\mathcal{O}_{\Gamma}$, in this case, $a^2-b^2\equiv 0 \ (mod~9).$
\end{enumerate}
\end{definition}

Now, let $p$ be a prime number. In the following Proposition, we give the decomposition of the prime $p$ in the pure cubic field $\Gamma=\mathbb{Q}(\sqrt[3]{ab^2})$. We denote by  $\mathcal{P}, \mathcal{P}_{i}$  prime ideals of $\Gamma$, and by $\mathcal{N}$ the absolute norm $\mathcal{N}_{\Gamma/\mathbb{Q}}$.

\begin{proposition}\label{pGAMMA}
$\\ $ Let $p$  a prime number such that $p\neq 3$, then:
\begin{enumerate}
\item  If $p$ divides $ab$ and $p\neq 3$, then $p\mathcal{O}_{\Gamma} = \mathcal{P}^3 , ~ \mathcal{N}(\mathcal{P}) = p$.
\item If $p \nmid 3ab$ and $p\equiv -1\ \ (mod\
3)$, then $p\mathcal{O}_{\Gamma}=\mathcal{P}\mathcal{P}_1$, with $\mathcal{N}(\mathcal{P}) = p$ and $\mathcal{N}(\mathcal{P}_1) =
p^2.$
\item If $p \nmid 3ab$ and $p \equiv 1 \ \
(mod\  3)$, then:
\begin{enumerate}
  \item $p\mathcal{O}_{\Gamma} = \mathcal{P}\mathcal{P}_1 \mathcal{P}_2$ with
 $\mathcal{N}(\mathcal{P})=\mathcal{N}(\mathcal{P}_1)=\mathcal{N}(\mathcal{P}_2)$, if $ab^2$ is a cubic residue modulo $p$; 
  \item $p\mathcal{O}_{\Gamma} = \mathcal{P}$ with
 $\mathcal{N}(\mathcal{P})=p^3$, if $ab^2$ is not a cubic residue modulo $p$.
\end{enumerate}

\end{enumerate}

\end{proposition}

\begin{proof}
See \cite{DED}.
\end{proof}
The ramification of the prime $3$ need a particular treatment, it is the purpose of the following Proposition:

\begin{proposition}\label{3GAMMA}
$ ~~\\ $ The decomposition into prime factors of $3$ is:
\begin{eqnarray*}
   3 \mathcal{O}_{\Gamma}= \left
    \lbrace
      \begin{aligned}
        \mathcal{P}^3 , \ if \ \ a^2 \not\equiv b^2  \ \ (mod\  9),\\
        \mathcal{P}^2 \mathcal{P}_1, \ if \ \ a^2 \equiv b^2 \ \ (mod\  9).\\
      \end{aligned}
    \right.
\end{eqnarray*}
\end{proposition}

\begin{proof}
See \cite{DED}.
\end{proof}

The ideal factorization rules for the $3$rd cyclotomic field $\mathrm{k_0}$ (see \cite{Clas}) is as follows:
\begin{itemize}
\item[$(i)$] $3\mathcal{O}_{\mathrm{k_0}}=\lambda^2=(1-\zeta_3)^2$;
\item[$(ii)$] $p\mathcal{O}_{\mathrm{k_0}}=\pi_1\pi_2$ in $\mathrm{k_0}$ if $p\equiv 1  \pmod  3$;
\item[$(iii)$] $q\mathcal{O}_{\mathrm{k_0}}=\mathsf{q}$ in $\mathrm{k_0}$ if $q\equiv -1  \pmod  3$.
\end{itemize}

Next, let $\mathrm{k}$ be the normal closure of  $\Gamma$. We note by $\mathcal{O}_{\mathrm{k}}$ the ring of integers of $\mathrm{k}$, $\mathfrak{P}$ and $\mathfrak{P}_s$ are prime ideals of $\mathrm{k},$ $\mathcal{N}=\mathcal{N}_{k/\mathbb{Q}}$ the norm of $\mathrm{k}$ on $\mathbb{Q}.$ Combining the ideal factorization rules for $\Gamma$ with those of the field $\mathrm{k_0}$. The decomposition of the prime $  3 $ in  $\mathrm{k}$  is the purpose of the following Theorem:

\begin{proposition}\label{3K}
$\\ $ The prime $ 3$  decomposes in $ \emph{k} $ as follows:
\begin{eqnarray*}
   3 \mathcal{O}_{\mathrm{k}}= \left
    \lbrace
      \begin{aligned}
        \mathfrak{P}^6 , \ si \ \ a^2 \not\equiv b^2  \ \ (mod\  9),\\
        \mathfrak{P}_1^2 \mathfrak{P}_2^2 \mathfrak{P}_3^2, \ si \ \ a^2 \equiv b^2 \ \ (mod\  9).\\
      \end{aligned}
    \right.
\end{eqnarray*}
\end{proposition}

\begin{proof}
 $\\ $ We have $3$ ramifies in the quadratic field $\mathrm{k}_0=\mathbb{Q}(\zeta_3)$.
 \begin{itemize}
\item[1)] Suppose that $\Gamma$ is the first kind,
 then by Proposition \ref{3GAMMA} we have $3\mathcal{O}_{\Gamma}=\mathcal{P}^3$.
 Hence,
$3\mathcal{O}_{k}=\mathfrak{P}^6$.
\item[2)] Conversely, suppose that $\Gamma$ is of second kind, 
then $3\mathcal{O}_{\Gamma}=\mathcal{P}^2\mathcal{P}_1$.
It follows that  $3\mathcal{O}_{k}=\mathfrak{P}_1^2\mathfrak{P}_2^2\mathfrak{P}_3^2$.
\end{itemize}
\end{proof}
However, we have the following Proposition in which we characterize the decomposition of prime ideals of $p\neq 3$ in $\mathrm{k}$.
\begin{proposition}\label{pK}
$\\ $ Let $p$  a prime number such that $p\neq 3$, then:
\begin{enumerate}
\item If $p$ divides $\mathcal{D}_{\Gamma}$, then:
\begin{enumerate}
\item $p\mathcal{O}_{\mathrm{k}}=\mathfrak{P}_1^3\mathfrak{P}_2^3$, with $\mathcal{N}(\mathfrak{P}_1)=\mathcal{N}(\mathfrak{P}_2)=p$, if and only if $-3$ is a quadratic residue modulo $p$.
  \item $p \mathcal{O}_{\mathrm{k}}=\mathfrak{P}^3$, with $\mathcal{N}(\mathfrak{P})=p^2$, if and only if $-3$ is not a quadratic residue modulo $p$.
\end{enumerate}
\item If $p$ does not divides  $\mathcal{D}_{\Gamma}$ and $p \equiv 1~(mod~3),$ then:
\begin{enumerate}
 \item $p$ decomposes completely  in $\emph{k}$ if and only if $\mathcal{D}_{\Gamma}$ is a cubic residue modulo $p$.
  \item $p \mathcal{O}_{\mathrm{k}}=\mathfrak{P}_1\mathfrak{P}_2$, with $\mathcal{N}(\mathfrak{P}_1)=\mathcal{N}(\mathfrak{P}_2)=p^3$, if and only if $\mathcal{D}_{\Gamma}$ is not a cubic residue modulo $p$.
\end{enumerate}
\item If $p$ does not divides  $\mathcal{D}_{\Gamma}$ and $p \equiv - 1~(mod~3),$ then:
 $p \mathcal{O}_{\mathrm{k}}=\mathfrak{P}_1\mathfrak{P}_2\mathfrak{P}_3$, with $\mathcal{N}(\mathfrak{P}_1)=\mathcal{N}(\mathfrak{P}_2)=\mathcal{N}(\mathfrak{P}_3)=p^2$, if and only if $-3$ is not a quadratic residue modulo $p$.
\end{enumerate}

\end{proposition}

\begin{proof}
\begin{enumerate}
\item We use Proposition \ref{pGAMMA} and the decomposition of prime ideals in the quadratic fields $\mathrm{k}_0=\mathbb{Q}(\zeta_3)$.

\item Suppose that $p$ does not divide $\mathcal{D}_{\Gamma}$ and $p\equiv 1~(mod~3)$, then $-3$ is a quadratic residue modulo $p$, then the multiplication formula gives 
 $$\left( \dfrac{-3}{p}\right)=\left( \dfrac{-1}{p}\right)\left( \dfrac{3}{p}\right).$$
  On the one hand, by the Euler's Theorem we have
   $$\left( \dfrac{-1}{p}\right)=(-1)^{(p-1)/2},$$
  On the other hand, the quadratic reciprocity law gives 
  $$\left( \dfrac{p}{3}\right)\left( \dfrac{3}{p}\right)=(-1)^{(p-1)/2},$$
    since $p\equiv 1~(mod~3)$, then $p$ is a square modulo $3$, which gives $\left( \dfrac{p}{3}\right)=1$, so 
$$\left( \dfrac{3}{p}\right)=(-1)^{(p-1)/2}.$$
   Then
   $$\left( \dfrac{-3}{p}\right)=((-1)^{(p-1)/2})^2=(-1)^{p-1}=1.$$
 Thus, $p$ decomposes completely in $\mathrm{k}_0$. 
 \begin{enumerate}
 \item If $\mathcal{D}_{\Gamma}$ is a cubic residue modulo $p$, 
then by Proposition \ref{pGAMMA} we have $p$ split completely in $\Gamma$. Hence $p$ split completely in $\mathrm{k}$.
 \item If $\mathcal{D}_{\Gamma}$ is not a cubic residue modulo $p$, we have $p$ remains prime in $\Gamma$. Hence $p\mathcal{O}_{\mathrm{k}}=\mathfrak{P}_1\mathfrak{P}_2$.
 \end{enumerate}
  \item We have $p\mathcal{O}_{\Gamma}=\mathcal{P}\mathcal{P}_1$, and $p$ remains inert in $\mathrm{k}_0$, hence the result.
 
\end{enumerate}
\end{proof}
\begin{remark}
In the preceding Proposition \ref{pK}, the situation $p\mathcal{O}_{\mathrm{k}}=\mathfrak{P}_1\mathfrak{P}_2$ is never happens because  if $p\equiv -1~(mod~3),$ we have  always $ \left( \dfrac{-3}{p}\right)=-1. $
\end{remark}


\section{\textbf{The generators of $C_{\mathrm{k},3}$} }\label{Sect3}

First, we let $C_{\mathrm{k},3}^{(\sigma)}=\lbrace \mathcal{A}\in C_{\mathrm{k},3} ~| ~ \mathcal{A}^{\sigma}=\mathcal{A} \rbrace$ be the group of ambiguous ideal classes of $\mathrm{k}/\mathrm{k}_0$, where $\sigma$ is a generator of $\operatorname{Gal}\left(\mathrm{k}/\mathrm{k}_0\right)$, and put $q^{*}=0$ or $1$ according to $\zeta_{3}$ is not norm or norm of
an element of $\mathrm{k}\backslash \{0\}$. Let $t$ be the
number of primes ramifies in $\mathrm{k}/\mathrm{k}_{0}$. Then according to \cite{GERTH1},
 we have
 $$|C_{\mathrm{k},3}^{(\sigma)}|=3^{t-2+q^{*}}.$$

 If we denote by $C_{\mathrm{k}_0,3}$ the Sylow $3$-subgroup of the ideal class group of $\mathrm{k}_0$,  $C_{\mathrm{k}_0,3}=\lbrace 1 \rbrace$. Let be $ C_{\mathrm{k},3}^{(1-\sigma)}=
 \lbrace \mathcal{A}^{(1-\sigma)}~|~ \mathcal{A} \in C_{\mathrm{k},3} \rbrace $.  By the exact sequence :
$$ 1 \longrightarrow C_{\mathrm{k},3}^{(\sigma)}  \longrightarrow C_{\mathrm{k},3} \overset{1-\sigma}{\longrightarrow } C_{\mathrm{k},3} \longrightarrow C_{\mathrm{k},3}/C_{\mathrm{k},3}^{1-\sigma} \longrightarrow 1 $$
we deduce that $$|C_{\mathrm{k},3}^{(\sigma)}|= | C_{\mathrm{k},3}/C_{\mathrm{k},3}^{1-\sigma} |.$$
 The fact that $C_{\mathrm{k},3}^{(\sigma)}$ and $C_{\mathrm{k},3}/C_{\mathrm{k},3}^{1-\sigma}$  are elementary abelian $3$-groups imply that:  
$$\operatorname{rank}\, C_{\mathrm{k},3}^{(\sigma)}= \operatorname{rank}\,( C_{\mathrm{k},3}/C_{\mathrm{k},3}^{1-\sigma}).$$
 Define the $3$-group $  C_{\mathrm{k},3}^{(1-\sigma)^{i}} $ for each $i\in \mathbb{N}$  by
  $$ C_{\mathrm{k},3}^{(1-\sigma)^{i}}=
 \lbrace \mathcal{A}^{(1-\sigma)^{i}}~|~ \mathcal{A} \in C_{\mathrm{k},3} \rbrace,$$ and let $s$ be the positive integer such that
 $C_{\mathrm{k},3}^{(\sigma)}\subseteq C_{\mathrm{k},3}^{(1-\sigma)^{s-1}} $
 and $C_{\mathrm{k},3}^{(\sigma)} \not \subseteq C_{\mathrm{k},3}^{(1-\sigma)^{s}} $.\\
 
 The following Proposition gives the structure of the $3$-class group $C_{\mathrm{k},3}$ when $27$ divides exactly the class number of $\mathrm{k}$:

\begin{proposition}\label{39}
Let be $\Gamma$ a pure cubic field,
\(\mathrm{k}\) its normal closure and \(u\) the index of units defined as above, then:
\begin{itemize}
\item[1)] 
$ C_{\mathrm{k},3}\simeq \mathbb{Z}/9\mathbb{Z}\times \mathbb{Z}/3\mathbb{Z}
\Leftrightarrow [ C_{\Gamma,3}\simeq \mathbb{Z}/9\mathbb{Z}\,\,
\text{and} \,\,u=1]$; 
\item[2)] $ C_{\mathrm{k},3}\simeq \mathbb{Z}/3\mathbb{Z}\times
\mathbb{Z}/3\mathbb{Z} \times \mathbb{Z}/3\mathbb{Z} \Leftrightarrow
[ C_{\Gamma,3}\simeq \mathbb{Z}/3\mathbb{Z}\times
\mathbb{Z}/3\mathbb{Z} \,\, \text{and} \,\,u=1]$.
\end{itemize}
\end{proposition}
\begin{proof}
\begin{itemize}
\item[1)] Assume that $C_{\mathrm{k},3}\simeq \mathbb{Z}/9\mathbb{Z}\times
\mathbb{Z}/3\mathbb{Z}$, then $h_{\mathrm{k},3}=27$. According to Theorem $14.1$ of \cite{B-C}, we have
$27=\frac{u}{3}\cdot h_{\Gamma,3}^{2}$, then $u=1$ because otherwise
$27$ will be a square in $\mathbb{N}$, which is a contradiction. Then
\(h_{\Gamma,3}^{2}=3^4\) and \(h_{\Gamma,3}=9\).\\
On the other hand, by Lemma $2.1$ and Lemma $2.2$ of \cite{GERTH2} we have
 \( C_{\mathrm{k},3}\simeq  C_{\Gamma,3}
 \times C_{\mathrm{k},3}^{-}\), then
  \( | C_{\mathrm{k},3}^{-}|=3\).
  Since \(C_{\mathrm{k},3}\) is of type
  \(\left(9,3\right)\), we deduce that \(C_{\mathrm{k},3}^{-}\) is a
  cyclic $3$-group of order $3$ and \(C_{\mathrm{k},3}^{+}\) is a cyclic $3$-group of
  order $9$. Therefore:
  \[u=1 \,\, \text{and}\,\,C_{\Gamma,3}\simeq \mathbb{Z}/9\mathbb{Z} .\]
  Reciprocally, assume that \(u=1\) and \(C_{\Gamma,3}\simeq
  \mathbb{Z}/9\mathbb{Z}\). By Theorem $14.1$ of \cite{B-C}, we deduce that
   \(|C_{\mathrm{k},3}|=
    \frac{1}{3}\cdot |C_{\Gamma,3}|^{2}
    \), then
    \( | C_{\mathrm{k},3}|=27\)
    and
    \( | C_{\mathrm{k},3}^{-}|=3\). Thus: \[C_{\mathrm{k},3}\simeq C_{\Gamma,3}
     \times C_{\mathrm{k},3}^{-} \simeq \mathbb{Z}/9\mathbb{Z}\times
\mathbb{Z}/3\mathbb{Z}.\]
 \item[2)] We have the same proof as above.
\end{itemize}
\end{proof}


\begin{lemma}\label{11}
 Let $\mathrm{k}=\mathbb{Q}(\sqrt[3]{p}, \zeta_3)$, where $p$ is a prime number
 such that $p\equiv 1~(\bmod~3)$. Let $C_{\mathrm{k},3}^{(\sigma)}$ be
the ambiguous ideal class group of $\mathrm{k}/\mathbb{Q}(\zeta_3) $, where $\sigma$ is a generator of $\operatorname{Gal}\left(\mathrm{k}/\mathbb{Q}(\zeta_3)\right)$. Then $\vert C_{\mathrm{k},3}^{(\sigma)}\vert =3$.
\end{lemma}

\begin{proof}
Since $p\equiv 1  \pmod  3$, then according to section \ref{subfactor}, we have $p=\pi_1\pi_2$, where $\pi_1$ and
$\pi_2$ are two primes of $\mathrm{k}_{0}$ such that
$\pi_2=\pi_1^{\tau}$ and $\pi_1 \equiv \pi_2 \equiv 1 ~(\bmod~3\mathcal{O}_{\mathrm{k}_0})$. We study all cases depending on the congruence class of $p$ modulo $9$, then:
\begin{itemize}
\item If $p \equiv 4\ \text{or} \ 7  \pmod  9$, then according to section \ref{subfactor}, the prime $3$ is  ramified in the field $L$, so the prime ideal $(1-\zeta_3)$ is  ramified in $\mathrm{k}/\mathrm{k}_0$.
Also $\pi_1$ and $\pi_2$ are totally ramified in $\mathrm{k}$. So $t=3$. In addition, the fact that $p \equiv 4\ \text{or} \ 7  \pmod  9$  imply that $\pi_{i}\not\equiv 1  \pmod { (1-\zeta_3)^3 }$ for $i= \lbrace 1, 2\rbrace$, then according to section $5$ of
\cite{GERTH1} we obtain
    $$\left(\frac{\zeta_3,p}{p} \right)_3 \neq 1$$
where the symbol $(\dfrac{ \ , \ }{})_3$ is the 
cubic Hilbert symbol. We deduce that $\zeta_3$ is not a norm in the
extension $\mathrm{k} / \mathrm{k}_{0} $, so $q^*=0$. Hence
$\operatorname{rank}\,   C_{\mathrm{k},3}^{(\sigma)}=1$ and then $\vert
C_{\mathrm{k},3}^{(\sigma)}\vert= 3 $.
\item If $p \equiv 1  \pmod  9$, the prime ideals which ramified in $\mathrm{k}
/ \mathrm{k}_0$ are $\pi_{1}$ and  $\pi_{2}$, so $t=2$. Moreover, $\pi_{1} \equiv \pi_2 \equiv 1  \pmod  {(1-\zeta_3)^3} $, then according to \cite{GERTH1}, the cubic Hilbert symbol: 
$$\left(\frac{\zeta_3,p}{ \pi_1}\right)_{3}=\left(\frac{\zeta_3,p}{ \pi_2}\right)_{3}=1,$$
 We conclude that $\zeta_3$ is a norm in the
extension $\mathrm{k} / \mathrm{k}_{0} $, then $q^*=1$, so $\operatorname{rank}\,   
C_{\mathrm{k},3}^{(\sigma)}=1$ and  $\vert
C_{\mathrm{k},3}^{(\sigma)}\vert= 3 $.
\end{itemize}

\end{proof}


The basic result for determining the generators of the $3$-class group of $\mathrm{k}=\mathbb{Q}(\sqrt[3]{p},\zeta_3)$  when the $3$-class number of $\mathrm{k}$ is divisible by $27$ exactly, where $p$ is a prime number such that $p \equiv 1 ~(mod~9)$,  is summarized in the following Theorem:

\begin{theorem}\label{Cara39}
Let $\Gamma= \mathbb{Q}(\sqrt[3]{p})$, where $p$ is a prime number such that $p \equiv 1 ~(mod~9)$, $\mathrm{k}=\mathbb{Q}(\sqrt[3]{p},\zeta_3)$ its normal closure and $C_{\mathrm{k},3}$ the $3$-class group of  $\mathrm{k}$. Assuming $9$ divides the $3$-class number of $\Gamma$ exactly, then:\\
The $3$-class group $C_{\mathrm{k},3}$ is isomorphic to $ \mathbb{Z}/9\mathbb{Z}\times
\mathbb{Z}/3\mathbb{Z}$  if and only if $u=1$.
\end{theorem}
\begin{proof}
\begin{itemize}
\item[$\Rightarrow)$] 
 By Proposition \ref{39}, it is clear that if $C_{\mathrm{k},3}$ is isomorphic to $ \mathbb{Z}/9\mathbb{Z}\times
\mathbb{Z}/3\mathbb{Z}$ then $u=1$.
\item[$\Leftarrow)$] Assume that $u=1$, then according to Theorem $14.1$ of \cite{B-C}, $h_{\mathrm{k},3}=27$. Since $9$ divides the $3$-class number of $\Gamma$, then by Lemma $5.11$ of \cite{Cal-Emer} we have $\operatorname{rank}\, C_{\mathrm{k},3}=2$. Hence $C_{\mathrm{k},3} \simeq \mathbb{Z}/9\mathbb{Z}\times
\mathbb{Z}/3\mathbb{Z}$.
\end{itemize}
\end{proof}


\begin{proposition}
\label{prop:carrr39}
Let $\Gamma= \mathbb{Q}(\sqrt[3]{p})$, where $p$ is a prime number such that $p \equiv 1 ~(mod~9)$, $\mathrm{k}=\mathbb{Q}(\sqrt[3]{p},\zeta_3)$ it's normal closure, and $C_{\mathrm{k},3}$ be the $3$-class group of $\mathrm{k}$. 
  Assume that $9$ divides the $3$-class number of $\Gamma$ exactly and $u=1$. Put $\langle \mathcal{A} \rangle=C_{\mathrm{k},3}^{+}$, where $\mathcal{A} \in C_{\mathrm{k},3}$ such that $\mathcal{A}^{9}=1$
 and $\mathcal{A}^{3} \neq 1$. Let  $C_{\mathrm{k},3}^{(\sigma)}$ be the $3$-group of ambiguous ideal classes of $\mathrm{k}/\mathrm{k}_0$ and   \(C_{\mathrm{k},3}^{1-\sigma}=\{ \mathcal{A}^{1-\sigma}\, \vert\,
\mathcal{A} \in C_{\mathrm{k},3} \}\) be the principal genus of $C_{\mathrm{k},3}$. Then:
 \begin{enumerate}
\item \(C_{\mathrm{k},3}^{(\sigma)}\) is a subgroup of \(C_{\mathrm{k},3}^{+}\),  $\mathcal{A} \not\in C_{\mathrm{k},3}^{(\sigma)} $ and
\(C_{\mathrm{k},3}^{(\sigma)}=\langle\mathcal{A}^{3}\rangle=\langle\mathcal{B}^{1-\sigma}\rangle\),
where $ \mathcal{B} \in C_{\mathrm{k},3}$ such that \(C_{\mathrm{k},3}^{-}=\langle\mathcal{B}\rangle\). 
 \item \(C_{\mathrm{k},3}^{-}=\langle(\mathcal{A}^{2})^{\sigma-1}\rangle\), and we have
\(C_{\mathrm{k},3}^{1-\sigma}= C_{\mathrm{k},3}^{-}\times
C_{\mathrm{k},3}^{(\sigma)} \) is a $3$-group of type \((3,3)\),
  \end{enumerate}
  where $C_{\mathrm{k},3}^{+}$ and $C_{\mathrm{k},3}^{-}$  are defined in Lemma $2.1$  of \cite{GERTH2}.
\end{proposition}
\begin{proof}
\begin{enumerate}
\item Since $9$ divides the $3$-class number of $\Gamma$ exactly and $u=1$, then according to Theorem \ref{Cara39}, \(C_{\mathrm{k},3}\) is of
type \((9,3)\), this implies by \cite{GERTH3} that the integer $s$ defined above is equal \(3\), and according to Case $4$ of \cite{GERTH3}, we conclude
that$\vert \left(C_{\mathrm{k},3}^{(\sigma)}
  \right)^{+}\vert =3$ and $\vert \left(C_{\mathrm{k},3}^{(\sigma)}
  \right)^{-}\vert= 1$,
  this implies that \(C_{\mathrm{k},3}^{(\sigma)}\) is a subgroup of
  \(C_{\mathrm{k},3}^{+}\). Therefore, \(\langle\mathcal{A}^{3}\rangle\) is the unique
  subgroup of order \(3\) of   \(C_{\mathrm{k},3}^{+}\) and \(C_{\mathrm{k},3}^{(\sigma)}\) is cyclic
  of order \(3\), then  \(C_{\mathrm{k},3}^{(\sigma)}=\langle\mathcal{A}^{3}\rangle\).\\ 
 Moreover, if
\(C_{\mathrm{k},3}^{-}=\langle\mathcal{B}\rangle\) where $ \mathcal{B} \in C_{\mathrm{k},3}$, then \(\mathcal{B} \not\in
C_{\mathrm{k},3}^{(\sigma)} \), so $\mathcal{B}^{\sigma} \neq \mathcal{B}$. Furthermore, $\mathcal{B}^{\sigma} \neq \mathcal{B}^2$ because otherwise we will have $\mathcal{B}^{\sigma^2}=(\mathcal{B}^2)^{\sigma}=(\mathcal{B}^{\sigma})^2=\mathcal{B}^{4}$, as $ \mathcal{B} \in C_{\mathrm{k},3}^{-}$, then $\mathcal{B}^3=1 $. Therefore, $\mathcal{B}^{\sigma^2}=\mathcal{B}$, so $\mathcal{B}^{\sigma^3}=\mathcal{B}^{\sigma}$, since $\sigma^3=1$, then $\mathcal{B}^{\sigma}=\mathcal{B}$. This is impossible because $\mathcal{B}^{\sigma}\neq \mathcal{B}$. As $\mathcal{B}^{3}=1$ and $\mathcal{B}^{1+\sigma+\sigma^2}=1$, then $\mathcal{B}^{\sigma^2}=\mathcal{B}^{2+2\sigma}$.  This equality makes it possible to show that $\mathcal{B}^{1-\sigma}$ is an ambiguous class. We conclude that
\(C_{\mathrm{k},3}^{(\sigma)}=\langle \mathcal{B}^{1-\sigma}\rangle\). 
\item We reason as in the assertion \(1\).
Since \(\mathcal{A}^{2}\not\in C_{\mathrm{k},3}^{-} \), we deduce that \(C_{\mathrm{k},3}^{-}=\langle(\mathcal{A}^{2})^{1-\sigma}\rangle\). then
\(C_{\mathrm{k},3}^{-}\) and \(C_{\mathrm{k},3}^{(\sigma)}\) are
contained in \(C_{\mathrm{k},3}^{1-\sigma}\) which is of order \(9\),
because \( | C_{\mathrm{k},3}|=27 \) and
\( | C_{\mathrm{k},3}^{(\sigma)}|=3\). Consequently,
\[C_{\mathrm{k},3}^{1-\sigma}=
C_{\mathrm{k},3}^{(\sigma)}\times
C_{\mathrm{k},3}^{-}=\langle\mathcal{A}^{3},\mathcal{B}\rangle.\]
\end{enumerate}
\end{proof}

Our principal result can be stated as follows:

\begin{theorem}
\label{thm:main} Let $\mathrm{k}=\mathbb{Q}(\sqrt[3]{ p}, \zeta_3 )$, where $p$ is a prime number such that $p \equiv 1  \pmod  9$. The prime $ 3$  decomposes in $ \emph{k} $ as $3\mathcal{O}_{\mathrm{k}}=
\mathcal{P}^{2}\mathcal{Q}^{2}\mathcal{R}^{2}$, where $\mathcal{P},
\,\,\mathcal{Q}$ and $\mathcal{R}$ are prime ideals of
$\mathrm{k}$ . Put \(h=\frac{h_{\mathrm{k}}}{27}\),
where \(h_{\mathrm{k}}\) is the class number of \(\mathrm{k}\).
Assume that $9$ divides exactly the $3$-class number of $\mathbb{Q}(\sqrt[3]{ p})$ and $u=1$. If $3$ is not a cubic residue  modulo $p$, then:
\begin{enumerate}
\item The class \([\mathcal{R}^{h}]\) generate $C_{\mathrm{k},3}^{+}$;
\item The \(3\)-class group \(C_{\mathrm{k},3}\)  is generated by classes \([\mathcal{R}^{h}]\) and
\([\mathcal{R}^{h}][\mathcal{P}^{h}]^2\), and we have:
\[C_{\mathrm{k},3}=\langle[\mathcal{R}^{h}]\rangle\times\langle[\mathcal{R}^{h}][\mathcal{P}^{h}]^2\rangle=\langle
[\mathcal{R}^{h}],[\mathcal{P}^{h}]^2\rangle.\]
\end{enumerate}
\end{theorem}

In Appendix of this paper, we illustrated this results by the numerical examples with the aid of \textit{Pari} programming \cite{PARI} and summarized in some tables in  section \ref{Exp}.
\begin{proof}
$\\ $ We start our proof by showing that $[\mathcal{R}^{h}]$ is of order $9$:\\
   Since the field $\Gamma=\mathbb{Q}(\sqrt[3]{ p} )$ with $p \equiv 1  \pmod  9$ is of second kind, then by Proposition \ref{3GAMMA} we have $3\mathcal{O}_{\Gamma}=
\mathcal{H}^{2}\mathcal{S}$, where $\mathcal{H}$ and $\mathcal{S}$ are prime of $\Gamma$, since $\mathcal{H}\mathcal{O}_{k}=\mathcal{P}\mathcal{Q}$ and $\mathcal{S}\mathcal{O}_{k}=\mathcal{R}^2$, then  $3\mathcal{O}_{\mathrm{k}}=
\mathcal{P}^{2}\mathcal{Q}^{2}\mathcal{R}^{2}$, where $\mathcal{P},
\,\,\mathcal{Q}$ and $\mathcal{R}$ are prime ideals of
$\mathrm{k}$.
Moreover, the prime ideal $\mathcal{R}$ is invariant by $\tau$, then
$[\mathcal{R}]\in \{ \chi\in C_{\mathrm{k},3} \vert
\chi^{\tau}=\chi\}$.\\
 If $9$ divides the $3$-class number of $\mathbb{Q}(\sqrt[3]{ p})$ exactly and $u=1$, then by Theorem \ref{Cara39} we have $C_{\mathrm{k},3}$ is of type $(9,3)$.
According to  Proposition \ref{39}, we have \(C_{\mathrm{k},3}^{+}\)
is cyclic of order \(9\), thus $[\mathcal{R}^{h}]^{9}=1$.
 Hence the class $[\mathcal{R}^{h}]$ is of order $9$ if and only if
 $\mathcal{R}^{h} \, \text{and}\, \mathcal{R}^{3h}$ are not principal.\\
We argue by the absurd: assume that $\mathcal{R}^{h}$ is principal,
we have
\begin{eqnarray*}
[\mathcal{R}^{h}]=1
& \Rightarrow & \exists \alpha \in
\mathrm{k} ~\mid ~ \mathcal{R}^{h}=\alpha\mathcal{O}_{\mathrm{k}}, \\
&\Rightarrow &  \mathcal{N}_{\mathrm{k}\vert \mathrm{k}_{0}}(\mathcal{R}^{h})=\mathcal{N}_{\mathrm{k}\vert \mathrm{k}_{0}}(\alpha \mathcal{O}_{\mathrm{k}} ),  \\
&\Rightarrow & \lambda^{h}\mathcal{O}_{\mathrm{k}_0}  =\mathcal{N}_{\mathrm{k}\vert \mathrm{k}_{0}}(\alpha) \mathcal{O}_{\mathrm{k}_0}, \ where \ \lambda=1-\zeta_3, \\
&\Rightarrow & \exists \epsilon \in E_{\mathrm{k}_{0}} ~|~
\lambda^{h}=\epsilon\cdot \mathcal{N}_{\mathrm{k}\vert
\mathrm{k}_{0}}(\alpha ), \\
&\Rightarrow & \exists \beta \in \mathcal{O}_{\mathrm{k}} ~|~
\lambda^{h}=\mathcal{N}_{\mathrm{k}\vert
\mathrm{k}_{0}}(\beta ),~ because ~ E_{\mathrm{k}_{0}}\subseteq
\mathcal{N}_{\mathrm{k}\vert \mathrm{k}_{0}}(\mathrm{k}^{*} ),
\end{eqnarray*}

that is to say $\lambda^{h}$ is a norm in
$\mathrm{k}=\mathrm{k}_{0}(\sqrt[3]{
p})=\mathrm{k}_{0}(\sqrt[3]{
\pi_{1}\pi_{2}})$, where $\pi_1$ and $\pi_2$ are two primes of
$\mathrm{k}_0$ such that $p=\pi_1\pi_2$. Hence, by property $(5)$
we have:
\[ (*)\,\,\,\,\,\,\,\,\,\left(\frac{\lambda^{h}, \pi_{1}\pi_{2}}{\mathcal{P}}\right)_{3}=1,\]
for all ideal $\mathcal{P}$ of
$\mathrm{k}_{0}$.\\
In particular, we calculate this symbol for $\mathcal{P}=\pi_1 \mathcal{O}_{\mathrm{k}_0}$ or $\mathcal{P}=\pi_2 \mathcal{O}_{\mathrm{k}_0}$.\\
For $\mathcal{P}=\pi_1 \mathcal{O}_{\mathrm{k}_0}$, using the  property $(1)$ of
the norm residue symbol, we have:
\[\left(\frac{\lambda^{h}, \pi_{1}\pi_{2}}{\mathcal{P}}\right)_{3}=
\left(\frac{\lambda^{h},
\pi_{1}\pi_{2}}{\pi_1}\right)_{3}=\left(\frac{\lambda^{h},
\pi_{1}}{\pi_1}\right)_{3}\cdot \left(\frac{\lambda^{h},
\pi_{2}}{\pi_1}\right)_{3}\]
the  properties $(2)$ and $(5)$ imply that:
\[\left(\frac{\lambda^{h},\pi_{2}}{\pi_1}\right)_{3}=\left(\frac{\lambda,\pi_{2}}{\pi_1}\right)_{1}^{h}=
 \left(\frac{\lambda}{\pi_1}\right)_{3}^{0\times h}=1.\]
and from the  properties $(1)$ and $(6)$ we have
  \[\left(\frac{\lambda^{h},
\pi_{1}}{\pi_1}\right)_{3}=\left(\frac{\lambda,
\pi_{1}}{\pi_1}\right)_{3}^{h}=\left(\frac{\lambda}{\pi_1}\right)_{3}^{h}\]
consequently
\[\left(\frac{\lambda^{h},
\pi_{1}\pi_2}{\pi_1}\right)_{3}=\left(\frac{\lambda}{\pi_1}\right)_{3}^{h}\].

Since the two primes $\pi_1$ and $\pi_2$ play  symmetric roles, then we obtain a similar relation when
$\mathcal{P}=\pi_2$:

$$\left(\frac{\lambda^{h},
\pi_{1}\pi_{2}}{\pi_2}\right)_{3}=\left(\frac{\lambda}{\pi_2}\right)_{3}^{h}.$$ \\
 The equation $(*)$ imply that
$$\left(\frac{\lambda}{\pi_1}\right)_{3}^{h}=\left(\frac{\lambda}{\pi_2}\right)_{3}^{h}=1.$$

The fact that $3$ is not a cubic residue  modulo $p$ imply that 
$$ \left(\frac{\lambda}{\pi_1\pi_2}\right)_{3} \neq 1  $$
then 
$$\left(\frac{\lambda}{\pi_1}\right)_{3} \neq 1  \ \text{or} \ \left(\frac{\lambda}{\pi_2}\right)_{3} \neq 1.$$
Since $3$ does not divide $h$, then 
$$\left(\frac{\lambda}{\pi_1}\right)_{3}^{h} \neq 1  \ \text{or} \ \left(\frac{\lambda}{\pi_2}\right)_{3}^{h} \neq 1.$$
which is a contradiction. Consequently, the
ideal $\mathcal{R}^{h}$ is not principal.\\

 Since the class $[\mathcal{R}^{h}]$ is invariant by $\tau$, we deduce that the ideal
 $\mathcal{R}^{3h}$ is  principal if and only if
 $\langle[\mathcal{R}^{h}]\rangle=C_{\mathrm{k},3}^{(\sigma)}$.\\
 Since $9$ divides exactly the $3$-class number of $\mathbb{Q}(\sqrt[3]{ p})$ and $u=1$, then by  we get $|C_{\mathrm{k},3}|=27$, so the positive integer $s$ defined above is equal \(3\), then $C_{k,3}^{(1-\sigma)^3}=1$, this implies that $C_{k,3}^{(\sigma)} = C_{k,3}^{(1-\sigma)^2}$. Suppose that $[\mathcal{R}^{h}]\in
C_{\mathrm{k},3}^{(\sigma)}$, then $[\mathcal{R}^h] = [ \mathcal{L}^{(1-\sigma)^2}]$ with $\mathcal{L}$ is prime ideal of $\mathrm{k}$,
 then there exist
$\alpha \in \mathrm{k}^{*}$ such that
$\mathcal{R}^{h}=(\alpha)\cdot\mathcal{L}^{(1-\sigma)^{2}}$, so $\mathcal{N}_{\mathrm{k}\vert \mathrm{k}_{0}}( \mathcal{R}^h )=\mathcal{N}_{\mathrm{k}\vert \mathrm{k}_{0}}( \alpha.\mathcal{L}^{(1-\sigma)^2})$, since $\mathcal{N}_{\mathrm{k}\vert \mathrm{k}_{0}}(\mathcal{L}^{(1-\sigma)^2})= \mathcal{L}^{(1-\sigma)(1-\sigma^3)} =1$, then $\lambda^{h}\mathcal{O}_{\mathrm{k}}=\mathcal{N}_{\mathrm{k}\vert
\mathrm{k}_{0}}(\alpha)\mathcal{O}_{\mathrm{k}}$, where $\lambda=1-\zeta_3$, so there exist $\varepsilon \in E_{\mathrm{k}}$ such that
$\lambda^{h}=\varepsilon\cdot \mathcal{N}_{\mathrm{k}\vert
\mathrm{k}_{0}}(\alpha )$, as $\lambda^{h} $ and $\mathcal{N}_{\mathrm{k}\vert
\mathrm{k}_{0}}(\alpha ) $ are in $\mathrm{k}_{0}$ then $\varepsilon \in E_{\mathrm{k}_{0}} $, since
$E_{\mathrm{k}_{0}}\subseteq \mathcal{N}_{\mathrm{k}\vert
\mathrm{k}_{0}}(\mathrm{k}^{*} )$ then $\lambda^{h}=
\mathcal{N}_{\mathrm{k}\vert \mathrm{k}_{0}}(\alpha_{1} ) $ where $\alpha_1
\in \mathcal{O}_{\mathrm{k}}$, that means $\lambda^{h}$ is a norm
in $\mathrm{k}=\mathrm{k}_{0}(\sqrt[3]{ p})$ which
is impossible. Finally, $[\mathcal{R}^{h}]$ is of order
$9$. This completes the proof of the first statement. \\

The second step in the proof is showing that the class $[\mathcal{R}^{h}][\mathcal{P}^{h}]^2$ is of order
$3$. We know that $(\mathcal{R}^{h})^{\tau}=\mathcal{R}^{h}$ and
$(\mathcal{P}^{h})^{\tau}=\mathcal{Q}^{h}$, then:
\begin{eqnarray*}
\left(\mathcal{R}^{h}\cdot
(\mathcal{P}^{h})^{2}\right)^{1+\tau} &=&
\left(\mathcal{R}^{h}\right)^{1+\tau} \cdot \left((\mathcal{P}^{h})^{2}\right)^{1+\tau}\\
&=&
(\mathcal{P}^{h})^{2}\cdot(\mathcal{R}^{h})^{2}\cdot
(\mathcal{Q}^{h})^{2}\\
&=& 3^{h}\mathcal{O}_{\mathrm{k}},
\end{eqnarray*}
 which imply that $[\mathcal{R}^{h}\cdot (\mathcal{P}^{h})^{2}]^{1+\tau}=1$. Hence
$[\mathcal{R}^{h}\cdot (\mathcal{P}^{h})^{2}]\in
C_{\mathrm{k},3}^{-}$.\\
On the other hand $\mathcal{R}^{h}\cdot (\mathcal{P}^{h})^{2}$ is
not principal, because otherwise we have
$[\mathcal{R}^{h}]=[\mathcal{P}^{h}]^{7}$, the fact that
$[(\mathcal{R}^{h})^{2}\cdot(\mathcal{P}^{h})^{2}\cdot(\mathcal{Q}^{h})^{2}]=1$
imply that $[(\mathcal{Q}^{h})^{2}]=1$, which is a contradiction
because the class $[\mathcal{Q}^{h}]$ is of order $9$ (reasoning as
$\mathcal{R}^{h}$). Hence
$[\mathcal{R}^{h}][\mathcal{P}^{h}]^{2}$ is of order $3$ and  generate the group  $C_{\mathrm{k},3}^{-}$.\\
Since $[\mathcal{R}^{h}]$ is a generator of
$C_{\mathrm{k},3}^{+}$, we deduce that
$$C_{\mathrm{k},3}=\langle[\mathcal{R}^{h}],[\mathcal{R}^{h}][\mathcal{P}^{h}]^{2}\rangle.$$
\end{proof}

\begin{corollary}
Using the same notation as above, we have the following properties:
\begin{enumerate}
\item $\mathcal{P}^{\sigma}=\mathcal{Q},\,\,\mathcal{Q}^{\sigma}=\mathcal{R}$;
\item $\mathcal{R}^{\tau}=\mathcal{R}\,\, \text{and}\,\, \langle[\mathcal{R}]\rangle=
\{ \chi \in C_{\mathrm{k},3} \vert \chi^{\tau }=\chi \}$;
\item $\mathcal{P}^{\tau\sigma}=\mathcal{P}\,\, \text{and}\,\, \langle[\mathcal{P}]\rangle=
\{ \chi \in C_{\mathrm{k},3} \vert \chi^{\tau \sigma}=\chi \}$;
\item $\mathcal{Q}^{\tau\sigma^{2}}=\mathcal{Q}\,\, \text{and}\,\, \langle[\mathcal{Q}]\rangle=
\{ \chi \in C_{\mathrm{k},3} \vert \chi^{\tau \sigma^{2}}=\chi \}$;
\item The \(3\)-class group can be generated also by:\\
$$C_{\mathrm{k},3}=\langle[\mathcal{P}^{h}],[\mathcal{P}^{h}][\mathcal{Q}^{h}]^2\rangle=
\langle[\mathcal{Q}^{h}],[\mathcal{Q}^{h}][\mathcal{R}^{h}]^2\rangle.$$
\item  The $3$-group $C_{\mathrm{k},3}^{(\sigma)}$ of ambiguous ideal classes is given by:\\
$$C_{\mathrm{k},3}^{(\sigma)}=\langle[\mathcal{R}^{3h}]\rangle=\langle[\mathcal{P}^{3h}]\rangle=\langle[\mathcal{Q}^{3h}]\rangle.$$
\item The principal genus \(C_{\mathrm{k},3}^{1-\sigma}=\{ \mathcal{A}^{1-\sigma}\, \vert\,
\mathcal{A} \in C_{\mathrm{k},3} \}\) is of type \((3,3)\) and
generated by:
$$C_{\mathrm{k},3}^{1-\sigma}=\langle[\mathcal{R}^{3h}],[\mathcal{R}^{h}][\mathcal{P}^{h}]^2\rangle.$$

\end{enumerate}
\end{corollary}
\begin{proof}
$\\ $
 The fact that the ideals \(\mathcal{P}^{h},\, \mathcal{Q}^{h}\) and
 \(\mathcal{R}^{h}\) are not principals, we prove the assertions
\((1), (2), (3) \) and \((4)\) by applying the decomposition 
of \(3\) in the normal closure \(\mathrm{k}\).\\
For the assertion \((5)\), since the ideals \(\mathcal{P}^{h},\,
\mathcal{Q}^{h}\) and \(\mathcal{R}^{h}\) are not principal, we obtain the result by the same reasoning above.\\
 The assertions \((6)\) and \((7)\) follows by using Proposition
 \ref{prop:carrr39}.
\end{proof}


\section{\textbf{ Appendix}}\label{Exp}\label{sect4}
Using the \textit{Pari} programming \cite{PARI}, we illustrate the results of
our main Theorem \ref{thm:main} by  numerical examples. We have
\[C_{\mathrm{k},3}=\langle[\mathcal{R}^{h}],[\mathcal{R}^{h}][\mathcal{P}^{h}]^2\rangle\]
The following table verifies, for each prime number $p \equiv 1 \pmod  9$ such that \(\left(\dfrac{3}{p}\right)_{3}\neq 1\) and $9$ divides the $3$-class number of $\mathbb{Q}(\sqrt[3]{ p})$ exactly and $u=1$, that the ideals $\mathcal{R}^{h}$ and $\mathcal{R}^{3h}$ are not principal. Therefore,
  the ideal $\mathcal{R}^{9h}$ is always principal.
$$ \text{Table 1}$$

\begin{longtable}{ | c | c | c | c | c |  }
 \hline
   $p$ & Type of $C_{\mathrm{k},3}$ &   Is principal $\mathcal{R}^{h}$  & Is principal $\mathcal{R}^{3h}$ & Is principal $\mathcal{R}^{9h}$     \\
\hline\hline

  $199$ & $[9, 3]$ & $[8, 0]~$ & $[6, 0]~$ & $[0, 0]~$  \\
  $487$ & $[9, 3]$ & $[10, 0]~$ & $[12, 0]~$ & $[0, 0]~$  \\
  $1297$ & $[9, 3]$ & $[16, 0]~$ & $[12, 0]~$ & $[0, 0]~$  \\
  $1693$ & $[9, 3]$ & $[2, 2]~$ & $[6, 0]~$ & $[0, 0]~$  \\
  $1747$ & $[9, 3]$ & $[8, 0]~$ & $[6, 0]~$ & $[0, 0]~$  \\
  $1999$ & $[9, 3]$ & $[8, 0]~$ & $[6, 0]~$ & $[0, 0]~$  \\
  $2017$ & $[9, 3]$ & $[8, 0]~$ & $[6, 0]~$ & $[0, 0]~$  \\
  $2143$ & $[9, 3]$ & $[14, 0]~$ & $[6, 0]~$ & $[0, 0]~$  \\
  $2377$ & $[9, 3]$ & $[7, 0]~$ & $[3, 0]~$ & $[0, 0]~$  \\
  $2467$ & $[9, 3]$ & $[20, 0]~$ & $[15, 0]~$ & $[0, 0]~$  \\
  $2593$ & $[9, 3]$ & $[4, 2]~$ & $[3, 0]~$ & $[0, 0]~$  \\
  $2917$ & $[9, 3]$ & $[8, 0]~$ & $[6, 0]~$ & $[0, 0]~$  \\
  $3511$ & $[9, 3]$ & $[10, 0]~$ & $[12, 0]~$ & $[0, 0]~$  \\
  $3673$ & $[9, 3]$ & $[8, 0]~$ & $[6, 0]~$ & $[0, 0]~$  \\
  $3727$ & $[9, 3]$ & $[5, 0]~$ & $[6, 0]~$ & $[0, 0]~$  \\
  $4159$ & $[9, 3]$ & $[4, 2]~$ & $[12, 0]~$ & $[0, 0]~$  \\
  $4519$ & $[9, 3]$ & $[4, 4]~$ & $[12, 0]~$ & $[0, 0]~$  \\
  $4591$ & $[9, 3]$ & $[1, 2]~$ & $[3, 0]~$ & $[0, 0]~$  \\
  $4789$ & $[9, 3]$ & $[25, 5]~$ & $[30, 0]~$ & $[0, 0]~$  \\
  $5347$ & $[9, 3]$ & $[8, 0]~$ & $[6, 0]~$ & $[0, 0]~$  \\
  $5437$ & $[9, 3]$ & $[77, 0]~$ & $[33, 0]~$ & $[0, 0]~$  \\
  $6949$ & $[9, 3]$ & $[7, 2]~$ & $[3, 0]~$ & $[0, 0]~$  \\
  $8209$ & $[9, 3]$ & $[2, 2]~$ & $[6, 0]~$ & $[0, 0]~$  \\
   $8821$ & $[9, 3]$ & $[4, 0]~$ & $[3, 0]~$ & $[0, 0]~$  \\

\hline

\end{longtable}
\label{tab:genor9}


However, we verify in the following table
 that the ideal $\mathcal{R}^{h}\mathcal{P}^{2h}$ is not principal and $\mathcal{R}^{h}\mathcal{P}^{2h}$ is of order $3$.


$$ \text{Table 2}$$

\begin{longtable}{| c | c | c | c |  }
 \hline
   $p$ & Type of $C_{\mathrm{k},3}$ &   Is principal $\mathcal{R}^{h}\mathcal{P}^{2h}$ & Is principal $(\mathcal{R}^{h}\mathcal{P}^{2h})^3$      \\
\hline \hline
 $199$ & $[9, 3]$    & $[0, 1]~$  & $[0, 0]~$ \\
 $487$ & $[9, 3]$    & $[0, 2]~$  & $[0, 0]~$ \\
 $1297$ & $[9, 3]$    & $[6, 4]~$  & $[0, 0]~$ \\

 $1693$ & $[9, 3]$    & $[6, 2]~$  & $[0, 0]~$ \\
 $1747$ & $[9, 3]$    & $[0, 1]~$  & $[0, 0]~$ \\
 $1999$ & $[9, 3]$    & $[0, 2]~$  & $[0, 0]~$ \\
 $2017$ & $[9, 3]$    & $[0, 2]~$  & $[0, 0]~$ \\
 $2143$ & $[9, 3]$    & $[0, 4]~$  & $[0, 0]~$ \\
 $2377$ & $[9, 3]$    & $[3, 2]~$  & $[0, 0]~$ \\
 $2467$ & $[9, 3]$    & $[0, 10]~$  & $[0, 0]~$ \\
 $2593$ & $[9, 3]$    & $[0, 2]~$  & $[0, 0]~$ \\
 $2917$ & $[9, 3]$    & $[0, 1]~$  & $[0, 0]~$ \\
 $3511$ & $[9, 3]$    & $[0, 2]~$  & $[0, 0]~$ \\
 $3673$ & $[9, 3]$    & $[0, 1]~$  & $[0, 0]~$ \\
 $3727$ & $[9, 3]$    & $[3, 1]~$  & $[0, 0]~$ \\
 $4159$ & $[9, 3]$    & $[6, 2]~$  & $[0, 0]~$ \\
 $4519$ & $[9, 3]$    & $[24, 4]~$  & $[0, 0]~$ \\
\hline

\end{longtable}
\label{tab:genor3}

\begin{quote}
Siham AOUISSI and Moulay Chrif ISMAILI \\ Department of Mathematics and Computer Sciences, \\ Mohammed 1st University,\\ Oujda - Morocco, \\
aouissi.siham@gmail.com, mcismaili@yahoo.fr.\\

Mohamed TALBI \\ Regional Center of Professions of Education and Training in the Oriental, \\ Oujda - Morocco, \\
ksirat1971@gmail.com.\\

Abdelmalek AZIZI \\Department of Mathematics and Computer Sciences,\\
 Mohammed 1st University, \\ Oujda - Morocco, \\
abdelmalekazizi@yahoo.fr.
\end{quote}

\end{document}